\documentclass{amsproc}

\usepackage{amsmath}
\usepackage{amssymb}
\usepackage{rotating}
\usepackage{graphicx}
\usepackage[tableposition=below]{caption}
\usepackage{bm}

\DeclareGraphicsExtensions{.png,.pdf,.eps}
\usepackage[all,2cell,cmtip]{xy}
\usepackage{tikz}
\usepackage{color}
\usepackage{longtable}
\usepackage{soul}
\newtheorem{theorem}{Theorem}[section]
\newtheorem{lemma}[theorem]{Lemma}
\newtheorem{corollary}[theorem]{Corollary}
\newtheorem{proposition}[theorem]{Proposition}
\newtheorem{conjecture}[theorem]{Conjecture}
\newtheorem{claim}[theorem]{Claim}

\theoremstyle{definition}

\newtheorem{definition}[theorem]{Definition}

\newtheorem{problem}[theorem]{Problem}
\newtheorem{remark}[theorem]{Remark}

\def\P{{\mathbb P}}

\def\R{{\mathbb R}}

\def\cM{{\mathcal M}}
\def\cN{{\mathcal{N}}}
\def\cO{{\mathcal{O}}}

\def\cU{{\mathcal U}}

\def\rat{\dashrightarrow}

\def\cOperatorname#1{\mathop{\rm #1}\nolimits}

\def\codim{\cOperatorname{codim}}

\def\deg{\cOperatorname{deg}}

\def\det{\cOperatorname{det}}

\def\rat{\cOperatorname{RatCurves}}

\def\ME{{\cOperatorname{ME}}}

\newcommand{\cME}[1]{\cOverline{\ME}}

\makeindex

\begin{document}
%\pagewiselinenumbers

\title{Positivity of the second exterior power of the tangent bundles}

\author{Kiwamu Watanabe}
\date{\today}
\address{Department of Mathematics, Faculty of Science and Engineering, Chuo University.
1-13-27 Kasuga, Bunkyo-ku, Tokyo 112-8551, Japan}
\email{watanabe@math.chuo-u.ac.jp}
\thanks{The author is partially supported by JSPS KAKENHI Grant Number 17K14153, the Sumitomo Foundation Grant Number 190170 and Inamori Research Grants.}


\begin{abstract} Let $X$ be a smooth complex projective variety with nef $\bigwedge^2 T_X$ and $\dim X \geq 3$. We prove that, up to a finite \'etale cover $\tilde{X} \to X$, the Albanese map $\tilde{X} \to {\rm Alb}(\tilde{X})$ is a locally trivial fibration whose fibers are isomorphic to a smooth Fano variety $F$ with nef $\bigwedge^2 T_F$. As a bi-product, we see that $T_X$ is nef or $X$ is a Fano variety. Moreover we study a contraction of a $K_X$-negative extremal ray $\varphi: X \to Y$. In particular, we prove that $X$ is isomorphic to the blow-up of a projective space at a point if $\varphi$ is of birational type. We also prove that $\varphi$ is a smooth morphism if $\varphi$ is of fiber type. As a consequence, we give a structure theorem of varieties with nef $\bigwedge^2 T_X$.
\end{abstract}

\maketitle

\section{Introduction} The classical uniformization theorem of Riemann surfaces says that every simply connected Riemann surface is biholomorphic to either the Riemann sphere, the complex plain or the open unit disc. As a higher-dimensional analogue, it is natural to ask the structures of compact K\"ahler manifolds with holomorphic bisectional curvature having some positivity properties. An algebro-geometric counterpart of this problem is to study smooth projective varieties whose tangent bundle satisfies certain positivity conditions. The story starts with the Frankel conjecture: the only compact K\"ahler manifold with positive holomorphic bisectional curvature is a projective space. It was proved by S. Mori \cite{Mori1} and Y. T. Siu-S. T. Yau \cite{SiuYau} independently. In a seminal paper \cite{Mori1}, S. Mori proved the stronger Hartshorne conjecture: a projective space is the only smooth projective variety whose tangent bundle is ample. Following \cite{Mori1, SiuYau}, N. Mok \cite{Mk0} solved the generalized Frankel conjecture, which gives a classification of compact K\"ahler manifolds with nonnegative holomorphic bisectional curvature. As an algebro-geometric analogue of Mok's result, F. Campana and T. Peternell \cite{CP1} studied smooth projective varieties with nef tangent bundle. In particular, they classified such varieties in the three dimensional case. After that, J. P. Demailly, T. Peternell and M. Schneider obtained a structure theorem of compact K\"ahler manifolds with nef tangent bundle:
\begin{theorem}[{\cite[Main Theorem]{DPS}}]\label{them:DPS:albanese}
Let $X$ be a compact K\"ahler manifold with nef tangent bundle $T_X$. Then there exists a finite \'etale cover $f: \tilde{X} \to X$ such that the Albanese morphism $\alpha_{\tilde{X}}: \tilde{X} \to {\rm Alb}(\tilde{X})$ is a locally trivial fibration with fiber $F$. Moreover $F$ is a smooth Fano variety with nef tangent bundle.    
\end{theorem} 
By this theorem, the study of compact K\"ahler manifolds with nef tangent bundle can be reduced to the case of Fano varieties. On the other hand, Campana and Peternell proposed the following conjecture:
\begin{conjecture}[{Campana-Peternell Conjecture \cite[11.2]{CP1}}]\label{conj:CP} Any smooth Fano variety with nef tangent bundle is homogeneous.
\end{conjecture}
This conjecture is widely open in general. For the recent development of the conjecture, we refer the reader to \cite{Kane19, MOSWW, Wat20} and references therein. One of the most significant results to approach the Campana-Peternell conjecture is the following result due to Demailly-Peternell-Schneider and L. E. Sol\'a Conde-J. A. Wi\'sniewski:
\begin{theorem}[{\cite[Theorem 5.2]{DPS}, \cite[Theorem~4.4]{1-ample}}]\label{them:DPS:SW} Let $X$ be a smooth projective variety with nef $T_X$ and $\varphi:X \to Y$ a contraction of a $K_X$-negative extremal ray $R$ of $\overline{NE}(X)$. Then $\varphi$ is a smooth morphism.
\end{theorem} 

Meanwhile, following \cite{CP1}, Campana and Peternell classified smooth projective $3$-folds with nef $\bigwedge^2 T_X$:
\begin{theorem}[{\cite{CP92}}]\label{them:CP:second:wedge}  Let $X$ be a smooth projective $3$-fold with nef $\bigwedge^2 T_X$. Then either $T_X$ is nef or $X$ is one of the following:
\begin{enumerate}
\item $X$ is the blow-up of $\P^3$ at a point, or 
\item $X$ is a Fano $3$-fold of index $2$ and $\rho_X=1$ except for those of degree $1$.
\end{enumerate}
\end{theorem} 
Moreover, in his preprint \cite{Yas12}, K. Yasutake classified Fano $4$-folds with nef $\bigwedge^2 T_X$ and $\rho_X>1$; and he \cite{Yas14} also proved that a smooth Fano variety $X$ with nef $\bigwedge^2 T_X$ is isomorphic to the blow-up of $\P^n$ at a point, provided that $X$ admits a birational elementary contraction and $\dim X \geq 3$. In \cite{Sch18}, D. Schmitz classified smooth toric varieties with nef $\bigwedge^2 T_X$ and $\dim X \geq 3$.
Related to these results, following the solution of the Hartshorne conjecture \cite{Mori1}, K. Cho and E. Sato \cite{CS} proved that a smooth projective variety $X$ with ample $\bigwedge^2 T_X$ is isomorphic to a projective space or a quadric.  Recently, D. Li, W. Ou and X. Yang \cite[Theorem~1.5]{LOY19} generalized this result for varieties $X$ with strictly nef $\bigwedge^2 T_X$.  

The purpose of this paper is to provide a structure theorem of smooth projective varieties with nef $\bigwedge^2 T_X$, which gives generalizations of some results in \cite{CP92, CS, LOY19, Sch18, Yas12, Yas14}. Our first result is an analogue of Theorem~\ref{them:DPS:albanese}:
\begin{theorem}\label{MT} 
Let $X$ be a smooth projective variety with nef $\bigwedge^2 T_X$ and $n=\dim X \geq 3$. Then there exists a finite \'etale cover $f: \tilde{X} \to X$ such that the Albanese morphism $\alpha_{\tilde{X}}: \tilde{X} \to {\rm Alb}(\tilde{X})$ is a locally trivial fibration whose fibers are isomorphic to a smooth Fano variety $F$. Moreover one of the following {{holds}}:
\begin{enumerate}
\item If $\dim {\rm Alb}(\tilde{X})>0$, then the tangent bundle $T_F$ is nef.    
\item If $\dim {\rm Alb}(\tilde{X})=0$, then $X \cong \tilde{X} \cong F$ and the second exterior power of the tangent bundle $\bigwedge^2 T_X$ is nef.  
\end{enumerate}  
\end{theorem} 

Most parts of this theorem directly follow from the recent result of J. Cao and A. H\"oring on the structure theorem for varieties with nef anticanonical divisor \cite{CH19}. The difficulty to prove Theorem~\ref{MT} is to show that $F$ is a Fano variety when $\dim {\rm Alb}(\tilde{X})=0$. This is a consequence of Theorem~\ref{them:rc:fano:0} below. 
As a corollary of Theorem~\ref{MT}, we obtain the following:

\begin{corollary}\label{cor:MT} Let $X$ be a smooth projective variety with nef $\bigwedge^2 T_X$ and $n=\dim X \geq 3$. Then $T_X$ is nef or $X$ is a Fano variety.
\end{corollary}

%Thanks to Theorem~\ref{MT}, up to a finite \'etale cover, the study of smooth projective varieties with nef $\bigwedge^2 T_X$ can be reduced to the case of Fano varieties. 
Our second result of this paper is an analogue of Theorem~\ref{them:DPS:SW}: 
\begin{theorem}\label{MT2}  Let $X$ be a smooth projective variety with nef $\bigwedge^2 T_X$ and  $n=\dim X \geq 3$. Then the following hold: 
\begin{enumerate}
\item If $X$ admits a birational contraction of a $K_X$-negative extremal ray, then $X$ is the blow-up of the projective space $\P^n$ at a point.  
\item If $X$ does not admit a birational contraction of a $K_X$-negative extremal ray, then any contraction of a $K_X$-negative extremal ray $\varphi: X\to Y$ is smooth. 
%Assume that $\dim Y >0$. If $\dim Y>1$, then $Y$ is also a smooth projective variety with nef $\bigwedge^2 T_Y$ which does not admit a birational contraction of a $K_Y$-negative extremal ray; and any fiber $F$ of $\varphi$ is a smooth Fano variety with nef tangent bundle and $\rho_F=1$.  
\end{enumerate}  
\end{theorem} 

%[{=Theorem~\ref{them:smooth:contraction}, Theorem~\ref{them:birational:contraction}, Proposition~\ref{prop:smooth:mor:target:blowup}}]

The first statement of Theorem~\ref{MT2} claims that Yasutake's result \cite[Main Theorem~1]{Yas14} holds for not only Fano varieties but also general varieties. Applying Theorem~\ref{MT2}, we prove the following:

\begin{theorem}\label{them:rc:fano:0} Let $X$ be a smooth projective variety with nef $\bigwedge^2 T_X$. If $X$ is a rationally connected variety with $n=\dim X \geq 3$, then $X$ is a Fano variety whose Kleiman-Mori cone $NE(X)$ is simplicial.
\end{theorem}

As an application of our main results, we obtain a structure theorem of varieties with nef $\bigwedge^2 T_X$ (see Theorem~\ref{them:further:study}). Applying the structure theorem, we generalize Schmitz's result \cite[Theorem~4]{Sch18} to any toroidal spherical variety: 
\begin{theorem}\label{them:Schmitz:generalization} Let $X$ be a smooth projective toroidal spherical variety with nef $\bigwedge^2 T_X$ and $n=\dim X \geq 3$. Then $X$ is isomorphic to a rational homogeneous variety or the blow-up of the projective space $\P^n$ at a point.
\end{theorem}

As the second application of the structure theorem, we shall prove that the analogue of Theorem~\ref{them:CP:second:wedge} holds for varieties of dimension at most six:
\begin{theorem}\label{them:CP:second:wedge:generalization}  Let $X$ be a smooth projective $n$-fold with nef $\bigwedge^2 T_X$. Assume that $4\leq n\leq 6$. Then either $T_X$ is nef or $X$ is one of the following:
\begin{enumerate}
\item $X$ is the blow-up of $\P^n$ at a point, or 
\item $X$ is a Fano variety of pseudoindex $n-1$ and $\rho_X=1$.
\end{enumerate}
\end{theorem}

\section{Preliminaries}

\subsection{Notation and Conventions}\label{subsec:NC} Along this paper, we work over the complex number field. We will use the basic notation and definitions in \cite{Har}, \cite{Kb}, \cite{KM}, \cite{L1} and \cite{L2}. For a projective variety $X$, we also use the following notation:
\begin{itemize}
\item We denote {the Picard number of $X$} by $\rho_X$.
\item For a $K_X$-negative extremal ray $R$ of $X$,  we denote the length by $\ell(R)$.
\item A {\it curve} means a projective variety of dimension $1$. 
\item For a smooth projective variety $F$, an {\it $F$-bundle} means a smooth morphism $f:Y \to X$ between smooth projective varieties whose fibers are isomorphic to $F$.
\item We use $\P^n$ to denote projective $n$-space and $Q^n$ to denote a smooth quadric hypersurface in $\P^{n+1}$.
\end{itemize}

\subsection{Properties of nef vector bundles}

Let $X$ be a smooth projective variety and $E$ a vector bundle on $X$. Then $E$ is {\it nef} if the tautological line bundle $\cO_{\P(E)}(1)$ is nef on $\P(E)$. We collect properties of nef vector bundles in the following:

\begin{proposition}[{\cite[Theorem~6.2.12]{L2} and \cite[Proposition~1.2 (8)]{CP1}
}]\label{prop:bundle:nef} Let $X$ be a smooth projective variety and $E$ a vector bundle on $X$. Then the following hold:
\begin{enumerate} 
\item If $E$ is nef, then so is any quotient bundle of $E$.
\item For a surjective morphism $f: Y \to X$, $E$ is nef if and only if $f^{\ast}E$ is nef.
\item If $E$ is nef, then so is any exterior power $\bigwedge^kE$. 
\item Let $0 \to F \to E \to G \to 0 $ be an exact sequence of vector bundles. If $F$ and $G$ are nef, then so is $E$.
\item Let $0 \to F \to E \to G \to 0 $ be an exact sequence of vector bundles. If $E$ is nef and the first Chern class $c_1(G)$ is numerically trivial, then $F$ is nef.
\end{enumerate}
\end{proposition}

\subsection{Extremal contractions} We frequently use the following two results:
\begin{proposition}[{Ionescu-Wi\'sniewski inequality \cite[Theorem~0.4]{Ion86}, \cite[Theorem~1.1]{Wis91} }]\label{prop:Ion:Wis} Let $X$ be a smooth projective variety and $\varphi: X \to Y$ a contraction of a $K_X$-negative extremal ray $R$ and let $E$ be its exceptional locus. Let $F$ be an irreducible component of a (non trivial) fiber of $\varphi$. Then
$$\dim E + \dim F \geq  \dim X + \ell(R)- 1.
$$
\end{proposition}

\begin{lemma}\label{lem:cone:fiber} Let $X$ be a smooth projective variety admitting a $K_X$-negative smooth contraction $f: X \to Y$ onto a simply connected smooth projective variety $Y$. We denote by $NE(X/Y)$ the convex cone in $N_1(X)$ generated by classes of curves contracted by $f$. For a fiber $F$ of $f$, consider the linear map $i_{\ast}: N_1(F) \to N_1(X)$ induced by the push-forward of $1$-cycles defined by the inclusion $i: F \hookrightarrow X$. Then the following hold:
\begin{enumerate}
\item The linear map $i_{\ast}: N_1(F) \to N_1(X)$ is injective.
\item $i_{\ast}(NE(F))=NE(X/Y)$.
\item $\rho_F=\rho_X-\rho_Y$.
\end{enumerate}
\end{lemma}

\begin{proof} See \cite[Lemma~3.3]{casa08} and \cite[comments after Remark~3.7]{casa08}.
\end{proof}

\begin{remark}\label{rem:rc:simply:conn} Any smooth projective rationally connected variety is simply connected (see for instance \cite[Corollary 4.18 (b)]{DebB}). Under the setting of Lemma~\ref{lem:cone:fiber}, assume additionally that $X$ is rationally connected; then so is $Y$. In particular, $Y$ is simply connected. In this paper, we only use Lemma~\ref{lem:cone:fiber} under this setting.  
\end{remark}

\subsection{Families of rational curves}
For a smooth Fano variety $X$, the minimal anticanonical degree of rational curves on $X$ is called the {\it pseudoindex} $\iota_X$:
$$
\iota_X:=\min \{ -K_X\cdot C \mid  C~\mbox{is a rational curve in} ~X\}.
$$
%In our proof, we shall use the following characterization results repeatedly:
\begin{theorem}[{\cite{CMSB, DH17, Ke, Mi}}]\label{them:CMSB:DH17} Let $X$ be a smooth Fano variety. 
\begin{enumerate}
\item If $\iota_X \geq \dim X+1$, then $X$ is isomorphic to a projective space. 
\item If $\iota_X = \dim X$, then $X$ is isomorphic to a smooth quadric. 
\end{enumerate}
\end{theorem}
As is well known, a result of Campana \cite{Cam92} and Koll\'ar-Miyaoka-Mori \cite{KMM} say that smooth Fano varieties are rationally connected (see also \cite[Chapter V. 2]{Kb}). Note that, according to \cite[Corollary~2.9]{KMM}, given a smooth morphism $\varphi: X \to Y$ between smooth projective varieties, if $X$ is a Fano variety, then so is $Y$. 
 
Let $X$ be a smooth projective variety.  A {\it family of rational curves} $\cM$ on $X$ means an irreducible component of $\rat^n(X)$. The family $\cM$ comes equipped with a $\P^1$-bundle $p: \cU \to \cM$ and an evaluation morphism $q: \cU \to X$. We refer the reader to \cite[Section~II.2]{Kb} for a detailed construction. 
A rational curve parametrized by $\cM$ is called an {\it $\cM$-curve}. A rational curve $C \subset X$ is {\it free} if for the normalization $f: \P^1 \to C \subset X$, $f^{\ast}T_X$ is nef. For a rational curve $C \subset X$, we denote by $[C] \in N_1(X)$ the numerical equivalence class of $C$. By abuse of notation, a point in $\rat^n(X)$ which corresponds to $C$ is also denoted by $[C] \in \rat^n(X)$. Since the family $\cM$ determines a numerical class, we denote it by $[\cM] \in N_1(X)$. The anticanonical degree of the family $\cM$ means the intersection number $\deg_{(-K_X)}\cM:=-K_X\cdot C$ for any curve $[C] \in \cM$. A half line spanned by $[\cM] \in N_1(X)$ is denoted by $R_{\cM}$. We denote by ${\rm Locus}(\cM)$ the union of all $\cM$-curves. For a point $x \in X$, the normalization of $p(q^{-1}(x))$ is denoted by $\cM_x$, and  by ${\rm Locus}(\cM_x)$ the union of all $\cM_x$-curves. 

\begin{definition}\label{def:dom:cov:unsplit} Under the above notation, 
\begin{enumerate}
\item $\cM$ is a {\it dominating family} (resp. {\it covering family}) if the evaluation morphism $q: \cU \to X$ is dominant (resp. surjective);
\item $\cM$ is a {\it minimal rational component} if it contains a free rational curve with minimal anticanonical degree;
\item $\cM$ is {\it locally unsplit} if for a general point $x\in {\rm Locus}(\cM)$, $\cM_x$ is proper;
\item $\cM$ is {\it unsplit} if $\cM$ is proper.
\end{enumerate}
\end{definition}
A family of rational curves $\cM$ is locally unsplit if $\cM$ is a dominating family with minimal degree with respect to some ample line bundle on $X$. A family of rational curves $\cM$ is a dominating family if and only if there {exists} a free $\cM$-curve (see for instance \cite[IV Theorem~1.9]{Kb}). 

\begin{lemma}\label{lem:degeneration:curves} Let $X$ be a smooth projective variety and $\cM\subset \rat^n(X)$ a family of rational curves. If $\cM$ is not proper, then there exists a rational $1$-cycle $Z=\sum_{i=1}^s a_i Z_i$ which satisfies the following:
\begin{enumerate}
\item $Z$ is algebraically equivalent to $\cM$-curves, where each $a_i$ is a positive integer and each $Z_i$ is a rational curve;
\item $a_1>1$ provided that $s=1$.
\end{enumerate}
 We call this rational $1$-cycle $Z=\sum_{i=1}^s a_i Z_i$ a {\it degeneration of $\cM$-curves}.
\end{lemma}

\begin{proof} By construction of $\rat^n(X)$, one has a morphism from $\cM$ to the Chow scheme of $1$-cycles on $X$: $h: \cM \to {\rm Chow}_1(X)$. The Chow scheme ${\rm Chow}_1(X)$ is projective \cite[I. Theorem~3.21.3]{Kb} and $h: \cM \to h(\cM)$ is finite; then $\cM$ is not closed in ${\rm Chow}_1(X)$, because $\cM$ is not proper. Thus there exists a $1$-cycle $Z\in \overline{\cM}\setminus \cM$. By \cite[II, Proposition~2.2]{Kb}, the $1$-cycle $Z$ can be written as $Z=\sum a_i Z_i$, where each $a_i$ is a positive integer and each $Z_i$ is a rational curve. 
\end{proof}

\begin{proposition}[{\cite[IV Corollary~2.6]{Kb}}]\label{prop:Ion:Wis:2} Let $X$ be a smooth projective variety and $\cM$ a locally unsplit family of rational curves on $X$. For a general point $x \in {\rm Locus}(\cM)$, 
$$\dim {\rm Locus}(\cM_x) \geq \deg_{(-K_X)}\cM+\codim_X{\rm Locus}(\cM) -1.
$$
\end{proposition}

Let $\cM$ be an unsplit covering family of rational curves on $X$. We say that two points $x_1, x_2 \in X$ are {\it $\cM$-equivalent} if these two points can be joined by a connected chain of $\cM$-curves. It is known that there exists a rationally connected fibration with respect to an unsplit family $\cM$:

\begin{theorem}[{\cite{Cam92, KMM} (see also \cite[Chapter~5]{DebB} and \cite[IV Theorem~4.16]{Kb})}]\label{them:RCF} There exists a nonempty open subset $X^0 \subset X$ and a projective morphism $\pi: X^0 \to Y^0$ whose fibers are $\cM$-equivalent classes.
\end{theorem}

\begin{definition}[{see for instance \cite[1.\!\!\! Introduction]{BCD}}] Let $X$ be a smooth projective variety and $\cM$ its unsplit covering family of rational curves. A morphism $\pi: X \to Y$ onto a normal projective variety is called a {\it geometric quotient} for $\cM$ if every fiber of $\pi$ is an $\cM$-equivalence class.
\end{definition}

\begin{theorem}[{\cite[Theorem~5.2]{DPS}, \cite[Theorem~4.4]{1-ample}, \cite[Theorem~2.2, 2.3]{Kane18}}]\label{them:sm:quot} Let $X$ be a smooth projective variety and $\cM$ its unsplit covering family of rational curves. If any $\cM$-curve is free, then there exists a geometric quotient for $\cM$. Moreover the quotient is a smooth morphism. 
\end{theorem}

\section{Properties of varieties with nef $\bigwedge^2 T_X$}

\subsection{Basic properties}

In this subsection, we collect basic results on varieties with nef $\bigwedge^2 T_X$. Although Lemmata~\ref{lem:-K:nef}, \ref{lem:nonfree} and \ref{lem:characteriz:AbVar} were contained in some papers such as \cite{CP92, Yas12, Yas14}, we give proofs for reader's convenience.

\begin{lemma}\label{lem:-K:nef} Let $X$ be a smooth projective variety with nef $\bigwedge^2 T_X$. Then $-K_X$ is nef.
\end{lemma}

\begin{proof} This follows from $\det \left(\bigwedge^2 T_X\right)=\omega_X^{1-n}$, where $n=\dim X$.
\end{proof}

\begin{lemma}[{\cite[Lemma~1.3]{CP92}, \cite[Lemma~2.9]{Yas12}}]\label{lem:nonfree} Let $X$ be an $n$-dimensional smooth projective variety with nef $\bigwedge^2 T_X$ and $n=\dim X \geq 3$. If $C$ is a non-free rational curve, then $-K_X\cdot C \geq n-1.$
\end{lemma}

\begin{proof} Denoting by $f: \P^1 \to X$ the normalization of $C$, there {exists} integers $a_i$ such that 
$$
f^{\ast}T_X \cong \bigoplus_{i=1}^n \cO_{\P^1}(a_i) \,\,\,\,(a_1\geq a_2 \geq \ldots \geq a_n).
$$  
By the assumption $C$ is not free, $a_n$ is negative. Since we have the natural inclusion $T_{\P^1} \to f^{\ast}T_X$, we have $a_1\geq 2$. Meanwhile, the nefness of $\bigwedge^2 T_X$ implies that $a_{n-1}+a_n \geq 0$. Hence $a_{n-1} \geq -a_n>0$. As a consequence, we have 
$$
-K_X\cdot C= \sum_{i=1}^n a_i=a_1+(a_2+ \ldots + a_{n-2})+(a_{n-1}+a_n) \geq 2+(n-3)+0=n-1.
$$ 
\end{proof}

\begin{proposition}\label{prop:fiber:target} 
Let $X$ be a smooth projective variety with nef $\bigwedge^2 T_X$. Assume that $\varphi: X \to Y$ is a smooth morphism with irreducible fibers. Then the following hold:
\begin{enumerate}
\item If $\dim Y\geq 2$, then $\bigwedge^2 T_Y$ is nef.
\item If $\dim Y\geq 1$, then any fiber $F$ of $\varphi$ admits nef $T_F$.
\end{enumerate}
\end{proposition} 

\begin{proof} We have an exact sequence
\begin{align*}
0 \to T_{X/Y} \to T_X \to \varphi^{\ast} T_Y\to 0 \tag{1}.
\end{align*}
Applying \cite[Chapter~II, Exercise~5.16 (d)]{Har}, we obtain the following exact sequences:
\begin{align*}
0 \to \bigwedge^2 T_{X/Y} \to E \to T_{X/Y}  \otimes \varphi^{\ast} T_Y\to 0 \tag{2}
\end{align*}
\begin{align*}
0 \to E \to \bigwedge^2T_X \to \varphi^{\ast} \left(\bigwedge^2T_Y\right)\to 0 \tag{3}
\end{align*} for some vector bundle $E$ on $X$. Since $\bigwedge^2 T_X$ is nef, so is $\varphi^{\ast} (\bigwedge^2T_Y)$ by Proposition~\ref{prop:bundle:nef} (1). Thus the first assertion follows from Proposition~\ref{prop:bundle:nef} (2). Restricting the above exact sequences (2) and (3) to the fiber $F$, one has the following exact sequences:
\begin{align*}
0 \to \bigwedge^2 T_{F} \to E|_{F} \to T_{F}^{\oplus \dim Y}\to 0 \tag{4}
\end{align*}
\begin{align*}
0 \to E|_{F} \to \left(\bigwedge^2T_X\right)|_{F} \to \varphi^{\ast}\left(\bigwedge^2T_Y\right)|_{F}\to 0 \tag{5}.
\end{align*}
If $\dim Y=1$, then $\varphi^{\ast}\left(\bigwedge^2T_Y\right)|_{F}=0$; thus it follows from the exact sequence (5) that $E|_{F} \cong \bigwedge^2\left(T_X\right)|_{F}$ is nef. If $\dim Y>1$, then $\varphi^{\ast}\left(\bigwedge^2T_Y\right)|_{F}\cong \cO_F^{\oplus\binom{\dim Y}{2}}$: thus it follows from the exact sequence (5) and Proposition~\ref{prop:bundle:nef} (5) that $E|_{F} $ is nef. As a consequence, in any case $E|_{F} $ is nef. Then the exact sequence (4) and Proposition~\ref{prop:bundle:nef} (1) concludes that $T_F$ is nef. 
\end{proof}

\begin{corollary}\label{cor:product:nef} Let $X$ be a product of positive-dimensional smooth projective varieties $Y$ and $Z$. If $\bigwedge^2 T_X$ is nef, then so is $T_X$. 
\end{corollary}

\begin{proof} Applying Proposition~\ref{prop:fiber:target} (2) to projections $p_1: X \to Y$ and $p_2: X \to Z$, the tangent bundles $T_Y$ and $T_Z$ are nef. Hence $T_X=p_1^{\ast}T_Y\oplus p_2^{\ast}T_Z$ is also nef.
\end{proof}

\begin{lemma}[{A special case of \cite[Theorem~1.1]{Yas12}}]\label{lem:characteriz:AbVar} Let $X$ be a smooth projective variety with nef $\bigwedge^2 T_X$. Assume that $\omega_X \cong \cO_X$. Then there exists a finite \'etale cover $f: \tilde{X} \to X$ such that $\tilde{X} $ is an Abelian variety.
\end{lemma}

\begin{proof} By Yau's Theorem \cite[Theorem~1]{Yau77}, $X$ admits a K\"ahler-Einstein metric. Then the result of Kobayashi \cite[Section~5.8]{Kob87} and L\"ubke \cite{Lub83} shows that the tangent bundle $T_X$ is $H$-semistable (in the sense of Mumford-Takemoto) with respect to any ample divisor $H$ on $X$. On the other hand, since $\bigwedge^2 T_X$ is numerically flat, the second Chern class $$c_2(\bigwedge^2 T_X)=\binom{n-1}{2}c_1^2(X)+(n-2)c_2(X)$$ is numerically trivial. This yields that $c_2(X)\cdot H^{n-2}=0$. Applying \cite[IV, Theorem~4.1]{Nak04}, $T_X$ is nef. Thus our assertion follows from \cite{DPS}. 
\end{proof}

\subsection{Families of minimal sections}

We often use the following notation:

\begin{definition}\label{def:section} Let $X$ be a smooth projective variety and $\varphi: X \to Y$ a $K_X$-negative contraction. Given a rational curve $\ell\subset Y$ with normalization $f: \P^1 \to \ell \subset Y$,  let $X_{\ell}$ be the fiber product $\P^1 \times_Y X$, and we denote by $\varphi_{\ell}$ the first projection $\P^1 \times_Y X\to \P^1$:
\[
  \xymatrix{
    X_{\ell} \ar[r]^i \ar[d]_{\varphi_{\ell}} &  X\ar[d]^{\varphi} \\
    \P^1  \ar[r]^f&  Y }
\]
 Since general fibers of $\varphi$ are smooth Fano varieties, these are rationally connected. Then the theorem of Graber-Harris-Starr \cite[Theorem~1.1]{GHS03} yields that $\varphi_{\ell}$ admits a section $\tilde{\ell} \subset X_{\ell}$. A section $\tilde{\ell}$ is a {\it minimal section of $\varphi_{\ell}$} if the anticanonical degree $\deg_{(-K_{X_{\ell}})}\tilde{\ell}$ is minimal among sections of $\varphi_{\ell}$. We denote by $\tilde{\ell}_X$ the image of $\tilde{\ell}$ by $i: X_{\ell} \to X$. 
 
Whereas a rational curve $C \subset X$ is called a {\it birational section of $\varphi$ over $\ell$} if $\varphi(C)=\ell$ and $\varphi|_C: C \to \ell$ is birational. In the above notation, $\tilde{\ell}_X$, which is the image of a minimal section of $\varphi_{\ell}$ by $i$, is a birational section of $\varphi$ over $\ell$. Moreover a birational section $C$ of $\varphi$ over $\ell$ is {\it minimal} if $\deg_{(-K_{X})}C$ is minimal among birational sections of $\varphi$ over $\ell$. Note that a minimal birational section of $\varphi$ over $\ell$ exists if the anticanonical degree of sections are bounded from below. In particular, if $\bigwedge^2T_X$ is nef, then by Lemma~\ref{lem:-K:nef} $-K_X$ is nef; thus a minimal birational section of $\varphi$ over $\ell$ exists.
\end{definition}

\begin{proposition}\label{proposition:over:P1} Let $X$ be a smooth projective variety with nef $\bigwedge^2 T_X$ and $n= \dim X \geq 3$. If $X$ admits a $K_X$-negative contraction onto $\P^1$, then $X$ is a product of $\P^1$ and a variety $Z$. In this case, $T_X$ is nef.  
\end{proposition}

\begin{proof} By \cite[Theorem]{CP92}, we may assume that $n \geq 4$. Let $\varphi: X \to \P^1$ be a $K_X$-negative contraction. We take a minimal section $\tilde{\ell} \subset X$ of $\varphi$ as in Definition~\ref{def:section} (In this case, $X_{\ell}=X$ if we put $\P^1 =\ell$). Let us take a family of rational curves $\cM \subset \rat^n(X)$ containing $[\tilde{\ell}]$; then we claim that  the family $\cM$ is unsplit.  
If $\cM$ were not unsplit, by Lemma~\ref{lem:degeneration:curves} we may find a rational $1$-cycle $Z=\sum a_i Z_i$ as a degeneration of $\cM$-curves, where each $a_i$ is a positive integer and each $Z_i$ is a rational curve. Then there exists a section of $\varphi$ among $Z_i$'s. This contradicts to the minimality of $\tilde{\ell}$. Thus $\cM$ is unsplit. 

For a general point $x \in {\rm Locus}(\cM)$, it follows from \cite[II Corollary~4.21]{Kb} that the restriction of $\varphi$ to ${\rm Locus}(\cM_x)$ is a finite morphism onto $\P^1$. This implies that $\dim {\rm Locus}(\cM_x)=1$. 
Meanwhile, it follows from Proposition~\ref{prop:Ion:Wis:2} that $\dim {\rm Locus}(\cM_x)\geq \deg_{(-K_X)} \cM-1$. If there exists a non-free $\cM$-curve, then by Lemma~\ref{lem:nonfree}, we have
$$1=\dim {\rm Locus}(\cM_x)\geq \deg_{(-K_X)} \cM-1\geq n-2.
$$
This contradicts to our assumption $n \geq 4$. Thus any $\cM$-curve is free. As a consequence, $\cM$ is an unsplit covering family such that any $\cM$-curve is free. Applying Theorem~\ref{them:sm:quot}, there exists a geometric quotient $\psi: X \to Z$ for $\cM$ and it is a smooth morphism. Since any $\cM$-curve is a section of $\varphi: X \to \P^1$, we see that $\varphi\times \psi: X \to \P^1 \times Z$ is bijective; then by Zariski's main theorem, we see that $\varphi\times \psi: X \to \P^1 \times Z$ is an isomorphism. The remaining part follows from Corollary~\ref{cor:product:nef}.  
\end{proof}

\section{Proof of Theorem~\ref{MT} and \ref{MT2}} 
\subsection{Weaker structure theorem of varieties with nef $\bigwedge^2 T_X$}
We begin with recalling the result of Cao and H\"oring on the structure theorem for varieties with nef anticanonical divisor:  
\begin{theorem}[\cite{CH19}]\label{them:CaoH} 
Let $X$ be a smooth projective variety with nef $-K_X$. Then there exists a finite \'etale cover $f: \tilde{X} \to X$ such that $\tilde{X}\cong Y \times Z$ where $\omega_Y\cong \cO_Y$ and the Albanese morphism $\alpha_{Z}:  Z \to {\rm Alb}(Z)$ is a locally trivial fibration such that the fiber $F$ is rationally connected. 
\end{theorem} 

By using this theorem, we shall prove a weaker version of Theorem~\ref{MT}:
\begin{proposition}\label{prop:weaker:MT} 
Let $X$ be a smooth projective variety with nef $\bigwedge^2 T_X$ and $n=\dim X \geq 3$. Then there exists a finite \'etale cover $f: \tilde{X} \to X$ such that the Albanese morphism $\alpha_{\tilde{X}}: \tilde{X} \to {\rm Alb}(\tilde{X})$ is a locally trivial fibration with fiber $F$. Moreover one of the following hold:
\begin{enumerate}
\item If $\dim {\rm Alb}(\tilde{X})>0$, then $F$ is a smooth Fano variety with nef tangent bundle.    
\item If $\dim {\rm Alb}(\tilde{X})=0$, then $X \cong \tilde{X} \cong F$ is a smooth rationally connected variety with nef $\bigwedge^2 T_{X}$.  
\end{enumerate}  
\end{proposition} 

\begin{remark}\label{rem:MT:Fano} In Proposition~\ref{prop:weaker:MT} (2), we do not claim that $F$ is a Fano variety. This is the only difference between Theorem~\ref{MT} and Proposition~\ref{prop:weaker:MT}. In Theorem~\ref{them:rc:fano} below, we shall prove that a smooth rationally connected variety with nef $\bigwedge^2 T_X$ is a Fano variety. 
\end{remark}

\begin{proof}[Proof of Proposition~\ref{prop:weaker:MT}] Let $X$ be a smooth projective variety with nef $\bigwedge^2 T_X$. By Theorem~\ref{them:CaoH}, there exists a finite \'etale cover $f: \tilde{X} \to X$ such that $\tilde{X}\cong Y \times Z$ where $\omega_Y\cong \cO_Y$ and the Albanese morphism $\alpha_{Z}:  Z \to {\rm Alb}(Z)$ is a locally trivial fibration such that the fiber $F$ is rationally connected. If $Y$ and $Z$ are positive-dimensional, then it follows from Corollary~\ref{cor:product:nef} that $T_X$ is nef. Thus our assertion follows from Theorem~\ref{them:DPS:albanese}. Hence we assume that one of varieties $Y$ and $Z$ is a point. If $\tilde{X}\cong Y$, then our assertion holds thanks to Lemma~\ref{lem:characteriz:AbVar}. Thus assume that $\tilde{X}\cong Z$. If $\dim {\rm Alb}(\tilde{X})=0$, then $\tilde{X}\cong F$ is rationally connected; thus so is $X$. Since any smooth projective rationally connected variety is simply connected (see \cite[Corollary 4.18 (b)]{DebB}), we have $X \cong \tilde{X}$. So consider the case $\dim {\rm Alb}(\tilde{X})>0$. In this case, by Proposition~\ref{prop:fiber:target} (2), $F$ is a rationally connected variety with nef $T_F$. Applying \cite[Proposition~3.10]{DPS}, we see that $F$ is a Fano variety.
\end{proof}

\subsection{Contractions of varieties with nef $\bigwedge^2 T_X$}

\begin{theorem}\label{them:smooth:contraction} Let $X$ be a smooth projective variety with nef $\bigwedge^2 T_X$ and $\varphi:X \to Y$ a contraction of a $K_X$-negative extremal ray $R$ of $\overline{NE}(X)$. If $n=\dim X \geq 3, \rho_X\geq 2$ and $\varphi$ is of fiber type, then $\varphi$ is a smooth morphism. 
\end{theorem}

We first prove a special case:

\begin{proposition}\label{prop:smooth:contraction:dim1} \rm Theorem~\ref{them:smooth:contraction} holds if $\dim Y=1$.
\end{proposition}

\begin{proof} We employ the notation as in the statement of Theorem~\ref{them:smooth:contraction}. By \cite[Corollary~3.15]{DebB}, $Y$ is an elliptic curve or a projective line $\P^1$. If $Y$ is $\P^1$, then our assertion follows from Proposition~\ref{proposition:over:P1}. Thus we assume that $Y$ is an elliptic curve. In this case, $\varphi: X \to Y$ factors through the Albanese map $\alpha_X: X \to {\rm Alb}(X)$:
\[
  \xymatrix{
    X \ar[r]^{\alpha_X\,\,\,\,} \ar[d]_{\varphi} &  {\rm Alb}(X)\ar[dl] \\
    Y  &  }
\]
Since $\varphi$ is a $K_X$-negative contraction, the morphism ${\rm Alb}(X) \to Y$ should be an isomorphism which in turn implies $\varphi$ is the Albanese map. Then by applying \cite[Theorem~1.2]{Cao16}, $\varphi$ is smooth as desired. 
\end{proof}

\begin{proof}[Proof of Theorem~\ref{them:smooth:contraction}] By Proposition~\ref{prop:smooth:contraction:dim1}, we may assume $\dim Y >1$. %Moreover, by \cite[Proposition~1.4]{CP92}, we may assume that $n>3$. 
Since general fibers of $\varphi:X \to Y$ are smooth Fano varieties, there {exists} a dominating family $\cM$ of $\rat^n(X)$ such that any $\cM$-curve is contracted by $\varphi$. By replacing if necessary, we may assume the anticanonical degree of the family $\cM$ is minimal among such families; then we claim that $\cM$ is locally unsplit. To prove this, fix an ample divisor $H$ on $Y$. Then, for sufficiently large $m \gg 0$, $-K_X+m\varphi^{\ast}H$ is ample and $\cM$ is a dominating family with minimal degree with respect to an ample divisor $-K_X+m\varphi^{\ast}H$, so that $\cM$ is locally unsplit. Applying Mori's bend and break lemma, one has $n +1 \geq  \deg_{(-K_X)}\cM\geq \ell(R)$. If moreover $\deg_{(-K_X)}\cM=n+1$, then it follows from Proposition~\ref{prop:Ion:Wis:2} that ${\rm Locus}(\cM_x)=X$ for a general point $x \in X$. Applying \cite[II Corollary~4.21]{Kb}, we see that $\rho_X=1$; this is a contradiction. Thus we obtain an inequality 
$$
n \geq  \deg_{(-K_X)}\cM\geq \ell(R).
$$
Here we claim the following:
\begin{claim}
$\deg_{(-K_X)}\cM= \ell(R)$.
\end{claim}
Assume the contrary, that is, $\deg_{(-K_X)}\cM> \ell(R)$. Choosing a rational curve $C$ in $R$ such that $\ell(R)=-K_X\cdot C$, $C$ is not free. In fact, if $C$ were free, then we could find a dominating family $\cM'$ of $\rat^n(X)$ such that any $\cM'$-curve is contracted by $\varphi$. However it contradicts to the minimality of the anticanonical degree of $\cM$. Thus by Lemma~\ref{lem:nonfree} we have $\deg_{(-K_X)}\cM=n$. Applying Proposition~\ref{prop:Ion:Wis:2}, for a general point $x \in X$, 
$$
\dim {\rm Locus}(\cM_x)\geq \deg_{(-K_X)}\cM-1=n-1. 
$$   
As a general fiber of $\varphi$ contains ${\rm Locus}(\cM_x)$, the relative dimension of $\varphi$ is $n-1$ which in turn implies $\dim Y=1$. This contradicts to our assumption that $\dim Y>1$.  As a consequence, we see that $\deg_{(-K_X)}\cM= \ell(R)$.

One can show that the family $\cM$ is unsplit. 
If not, by Lemma~\ref{lem:degeneration:curves} we could find a rational $1$-cycle $Z=\sum_{i=1}^s a_i Z_i$ as a degeneration of $\cM$-curves, where each $a_i$ is a positive integer and each $Z_i$ is a rational curve. Then it follows from the extremality of the ray $R$ that each $[Z_i]$ is contained in $R$. However this contradicts to the minimality of $C$. Thus the family $\cM$ is an unsplit covering family. Being $\cM$ unsplit, we may consider a rationally connected fibration $\varphi_{\cM}: X {\cdots} \to Y'$ with respect to $\cM$. If any $\cM$-curve is free, then, applying Theorem~\ref{them:sm:quot}, we see that $\varphi_{\cM}: X \to Y'$ is the geometric quotient of $X$ for $\cM$ and it is a smooth morphism. By construction, the quotient morphism $\varphi_{\cM}$ is nothing but $\varphi$. Hence in this case our assertion holds.
Thus it is enough to show the following:
\begin{claim} 
Any $\cM$-curve is free.
\end{claim}
To prove this, assume the contrary; there would exist a non-free $\cM$-curve $C_0$. By Lemma~\ref{lem:nonfree}, one has $\ell(R)=-K_X\cdot C_0 \geq n-1$. 
Applying Proposition~\ref{prop:Ion:Wis:2}, for a general point $x \in X$, 
$$
\dim {\rm Locus}(\cM_x)\geq -K_X\cdot C_0-1\geq n-2. 
$$   
Thus the relative dimension of $\varphi_{\cM}$ is at least $n-2$. By \cite[Theorem~1 and its proof, Theorem~2]{BCD}, we see that $\varphi_{\cM}: X \to Y'$ is the geometric quotient and it is equidimensional. Being $\varphi_{\cM}: X \to Y'$ a contraction of a geometric extremal ray $R_{\cM}$, $\varphi_{\cM}$ is nothing but $\varphi$. In particular, $\varphi$ is equidimensional and one has inequalities
$$
n-1\geq n-\dim Y\geq \ell(R)-1\geq n-2.
$$
If $n-\dim Y= \ell(R)-1$, then \cite[Theorem~1.3]{HNov13} tells us that $\varphi: X \to Y$ is a projective bundle. This contradicts to the existence of a non-free rational curve $C_0$. Hence we obtain 
$\dim Y =1$; however this contradicts to our assumption that $\dim Y >1$. 
\end{proof}

\begin{remark}\label{rem:contraction:unsplit} Let $X$ be a smooth projective variety with nef $\bigwedge^2 T_X$ and $\varphi:X \to Y$ a contraction of a $K_X$-negative extremal ray $R$ of $\overline{NE}(X)$. Assume that  $n=\dim X \geq 3, \rho_X\geq 2$ and $\varphi$ is of fiber type. Then we may choose a free rational curve $C$ of minimal anticanonical degree among those spanning the corresponding ray $R$ and take a family of rational curves $\cM \subset \rat^n(X)$ containing $[C]$. Then as in the proof of Theorem~\ref{them:smooth:contraction}, one can check that $\cM$ is unsplit. 
\end{remark}

\begin{corollary}\label{cor:geom:quot:unsplit} Let $X$ be a smooth projective variety with nef $\bigwedge^2 T_X$ and $n= \dim X \geq 3$. If $X$ admits an unsplit covering family of rational curves $\cM$, then there exists a smooth geometric quotient $\varphi: X \to Y$ for $\cM$. 
\end{corollary}

\begin{proof} If $\deg_{(-K_X)}\cM\leq n-2$, then our assertion follows from Lemma~\ref{lem:degeneration:curves},
 Lemma~\ref{lem:nonfree} and Theorem~\ref{them:sm:quot}; thus suppose $\deg_{(-K_X)}\cM> n-2$. Then, for a general point $x \in X$, by Proposition~\ref{prop:Ion:Wis:2} we have $\dim {\rm Locus}(\cM_x)\geq n-2$. Applying \cite[Theorem~1 and 2]{BCD}, we see that there exists a geometric quotient $\varphi: X \to Y$ for $\cM$, which is a contraction of a $K_X$-negative extremal ray $R_{\cM}$. The smoothness of $\varphi$ follows from Theorem~\ref{them:smooth:contraction}. 
\end{proof}

\begin{theorem}\label{them:birational:contraction} Let $X$ be a smooth projective variety with nef $\bigwedge^2 T_X$ and $\varphi:X \to Y$ a contraction of a $K_X$-negative extremal ray $R$ of $\overline{NE}(X)$. If $n=\dim X \geq 3, \rho_X\geq 2$ and $\varphi$ is of birational type, then $X$ is isomorphic to the blow-up of the projective space $\P^n$ at a point.
\end{theorem}

\begin{proof} Let $E$ be an irreducible component of the $\varphi$-exceptional locus and $F$ an irreducible component of any nontrivial fiber of $\varphi$. Then Proposition~\ref{prop:Ion:Wis} shows that 
$$
2(n-1)\geq \dim E+\dim F\geq n+\ell(R)-1.
$$ By Lemma~\ref{lem:nonfree}, one has $\ell(R) \geq n-1$ which in turn implies $\dim E=\dim F=n-1$ and $\ell(R)=n-1$. This means that $\varphi$ is a divisorial contraction such that $\varphi (E)$ is a point. This argument tells us that any birational contraction of a $K_X$-negative extremal ray is a divisorial contraction which contracts its exceptional divisor to a point. 

Let us remark that $X$ is uniruled. In fact one has an \'etale cover $\tilde{X}$ of $X$ as in Proposition~\ref{prop:weaker:MT}. By the existence of a $K_X$-negative extremal ray, the Albanese map $\alpha_X: \tilde{X} \to {\rm Alb}(\tilde{X})$ is a nontrivial Fano fibration; thus $\tilde{X}$ is uniruled. This yields that $X$ is also uniruled. 
Hence there exists a free rational curve on $X$. Thus we may take a minimal rational component $\cM\subset \rat^n(X)$. We claim that $\cM$ is unsplit. To prove this, assume the contrary; if not, by Lemma~\ref{lem:degeneration:curves}, we might find a rational $1$-cycle $Z=\sum_{i=1}^s a_i Z_i$ as a degeneration of $\cM$-curves, where each $a_i$ is a positive integer and each $Z_i$ is a rational curve. Since $Z_i$ is not free thanks to the minimality of the anticanonical degree of $\cM$, Lemma~\ref{lem:nonfree} implies that $-K_X\cdot Z_i\geq n-1$. Thus one obtains
$$
-K_X\cdot Z=\sum a_i \left(-K_X\cdot Z_i\right)\geq 2(n-1).
$$
By the same way as in the proof of Theorem~\ref{them:smooth:contraction}, Mori's bend and break lemma and \cite[II Corollary~4.21]{Kb} yield that $-K_X\cdot Z$ is at most $n$. This contradicts to our assumption that $n\geq 3$.

Applying Corollary~\ref{cor:geom:quot:unsplit}, we obtain a smooth geometric quotient $\psi: X \to Z$ for $\cM$. 
Since any fiber of $\psi$ should has dimension at most one, $\psi$ is a $\P^1$-bundle. By \cite[Proposition~2.2 and its proof]{Fuj12}, $\psi: X\cong \P(\cO_{Z}\oplus \cO_Z(m)) \to Z$ and $E$ is its section, where $Z$ is an $(n-1)$-dimensional Fano manifold of $\rho_Z=1$ and $\cN_{E/X}\cong \cO_E(m)~(m<0)$. For any rational curve $D \subset E \cong Z$, it is not free as a curve on $X$; thus by Lemma~\ref{lem:nonfree} 
$$-K_E\cdot D=-K_X|_E\cdot D-E\cdot D\geq (n-1)-m\cO_Z(1)\cdot D \geq n.
$$ 
By Theorem~\ref{them:CMSB:DH17} (1), $E \cong Z$ is isomorphic to $\P^{n-1}$. Moreover taking a line on $E \cong \P^{n-1}$ as a rational curve $D$ as in the above inequality, we obtain $m=-1$. As a consequence, $X$ is isomorphic to $\P(\cO_{\P^{n-1}}\oplus \cO_{\P^{n-1}}(-1))$, that is, $X$ is isomorphic to the blow-up of the projective space $\P^n$ at a point as desired. 
\end{proof}

\begin{proposition}\label{prop:smooth:mor:target:blowup} Let $X$ be a smooth projective variety with nef $\bigwedge^2 T_X$. If $n=\dim X \geq 3$, then $X$ does not admit a $K_X$-negative smooth contraction onto the blow-up of the projective space $\P^m$ at a point.
\end{proposition}

\begin{proof} To prove our assertion, let us assume the contrary; then there exists a $K_X$-negative smooth contraction $\varphi: X \to Y\cong \P(\cO_{\P^{m}}\oplus \cO_{\P^{m}}(-1))$. We denote by $\beta: Y \to \P^{m+1}$ the blow-up of $\P^{m+1}$ at a point and by $E$ its exceptional divisor, which is a section $E=\P(\cO_{\P^{m}}(-1))$ of a $\P^1$-bundle $\P(\cO_{\P^{m}}\oplus \cO_{\P^{m}}(-1)) \to \P^m$. Denoting by $\ell$ a line on $E\cong \P^m$, let us consider $\varphi_{\ell}: X_{\ell} \to \ell=\P^1$ and its section $\tilde{\ell} \subset X_{\ell}$ as in Definition~\ref{def:section}. Note that $X_{\ell}$ coincides with $\varphi^{-1}(\ell)$; thus $\tilde{\ell}=\tilde{\ell}_X$ under the notation as in Definition~\ref{def:section}. 
We have an exact sequence of normal bundles:
\begin{align*}
0 \to N_{\tilde{\ell}/X_{\ell}} \to N_{\tilde{\ell}/X} \to N_{X_{\ell}/X}|_{\tilde{\ell}} \to 0. \tag{1}
\end{align*}
By \cite[II, Proposition~8.10]{Har}, we obtain 
\begin{align*}
N_{X_{\ell}/X}|_{\tilde{\ell}}\cong \varphi_{\ell}^{\ast}(N_{\ell/Y})|_{\tilde{\ell}}\cong \cO_{\P^1}(1)^{\oplus m-1}\oplus \cO_{\P^1}(-1). \tag{2}
\end{align*}
By (1) and (2), we obtain
\begin{align*}
-K_{X_{\ell}}\cdot \tilde{\ell}=-K_X\cdot \tilde{\ell}+(2-m)\geq 2-m.\tag{3}
\end{align*}
Since $-K_{X_{\ell}}\cdot \tilde{\ell}$ is bounded from below, by replacing if necessary, we may assume that $\tilde{\ell}$ is a minimal section of $\varphi_{\ell}$. We then take a family of rational curves $\cN \subset \rat^n{X_{\ell}}$ containing $[\tilde{\ell}]$. According to the minimality of $\deg_{(-K_{X_{\ell}})}\tilde{\ell}$, by the same way as in the proof of Proposition~\ref{proposition:over:P1}, we can check that $\cN$ is unsplit. Then we claim that $-K_X\cdot \tilde{\ell} \leq n-2$. If not, the inequality (3) would imply 
\begin{align*}
-K_{X_{\ell}}\cdot \tilde{\ell}=-K_X\cdot \tilde{\ell}+(2-m)\geq (n-1)+(2-m)=\dim X_{\ell}+1.\tag{4}
\end{align*}
By Proposition~\ref{prop:Ion:Wis:2}, for a general point $x \in X_{\ell}$, ${\rm Locus}(\cN_x)=X_{\ell}$. Then applying \cite[II Corollary~4.21]{Kb}, we see that $\rho_{X_{\ell}}=1$; this is a contradiction. Thus we have $-K_X\cdot \tilde{\ell} \leq n-2$; then by Lemma~\ref{lem:nonfree}, $\tilde{\ell}$ is a free rational curve in $X$. Since $\varphi$ is smooth, this implies that $\ell$ is also free; however this is a contradiction, because $\ell$ is contained in the exceptional divisor $E$ of the blow-up $\beta: Y \to \P^{m+1}$. 
\end{proof}

By applying the same method as in the proof of Proposition~\ref{prop:smooth:mor:target:blowup}, we can prove the following:
\begin{proposition}\label{prop:target:surface} Let $X$ be a smooth rationally connected projective variety with nef $\bigwedge^2 T_X$ and $n= \dim X \geq 3$. If $X$ admits a $K_X$-negative smooth contraction $\varphi: X \to Y$ onto a projective surface $Y$, then $Y$ is isomorphic to $\P^2$ or $\P^1\times \P^1$. %In particular, $T_Y$ is generated by global sections.  
\end{proposition}

\begin{proof} To prove our assertion, assume the contrary, that is, assume $Y$ is not isomorphic to $\P^2$ or $\P^1\times \P^1$. Since $X$ is rationally connected, so is $Y$. By Proposition~\ref{prop:fiber:target}, $-K_Y$ is nef. Then it follows from \cite[Proposition~1.1]{CP92} that $Y$ contains a curve $\ell$ with negative self-intersection number ${\ell}^2<0$. For $\varphi: X \to Y$ and $\ell$, we can take a minimal section $\tilde{\ell} \subset X_{\ell}$ as in Definition~\ref{def:section}; then we can prove that $\tilde{\ell}$ is free in $X$; thus so is $\ell$ in $Y$. This is a contradiction.
\end{proof}

\begin{proposition}\label{prop:Y} Let $X$ be a smooth projective variety with nef $\bigwedge^2 T_X$. If $X$ is a rationally connected variety with $n=\dim X\geq 3$, then there exists a smooth $K_X$-negative contraction $\varphi: X \to Y$ satisfying one of the following:
\begin{enumerate}
\item $Y$ is isomorphic to $\P^1$;
\item $Y$ is isomorphic to $\P^2$;
\item $\dim Y \geq 3$ and $\rho_Y=1$.
\end{enumerate}
\end{proposition}

\begin{proof} We prove this by induction on the Picard number $\rho_X$. By Mori's cone theorem and Theorem~\ref{them:smooth:contraction}, we have a smooth contraction of a $K_X$-negative extremal ray $\varphi_1: X \to X_1$. If $\dim X_1=1$, then $X_1$ is $\P^1$; thus there is nothing to prove. If $\dim X=2$, then by Proposition~\ref{prop:target:surface} $X_1$ is $\P^2$ or $\P^1 \times \P^1$. In the former case, our claim holds. In the later case, the composition of $\varphi_1$ and a projection $p_1: \P^1 \times \P^1 \to \P^1$ is a smooth contraction. It follows from Proposition~\ref{prop:fiber:target} and \cite[Chapter 0, Proposition]{DPS} that $p_1\circ \varphi_1$ is a $K_X$-negative contraction. Thus we assume that $\dim X_1\geq 3$. If $\rho_{X_1}=1$, our claim holds. So we assume that $\rho_{X_1}>1$. Then Proposition~\ref{prop:fiber:target} implies that $X_1$ is a smooth rationally connected projective variety with nef $\bigwedge^2 T_{X_1}$. The inductive assumption tells us that there exists a smooth $K_{X_1}$-negative contraction $\varphi_2: X_1 \to Y$ such that $Y$ satisfies one of (1)-(3). By Proposition~\ref{prop:fiber:target} and \cite[Chapter 0, Proposition]{DPS}, the composition $\varphi_2\circ \varphi_1: X \to Y$ is a $K_X$-negative smooth contraction. As a consequence, our claim holds.
\end{proof}

\subsection{A geometric quotient for a family of minimal sections}
We start with setting up our notation:
Let $X$ be a smooth projective variety with nef $\bigwedge^2 T_X$ and $n=\dim X \geq 3$. Assume that $X$ admits a smooth $K_X$-negative contraction $\varphi: X \to Y$ onto a smooth Fano variety $Y$ with $\dim Y\geq 2$ and $\rho_Y=1$. Let us fix a rational curve $\ell \subset Y$ such that $-K_Y\cdot \ell=\iota_Y$. By Definition~\ref{def:section}, we may construct a minimal birational section of $\varphi$ over $\ell$, which is denoted by $\tilde{\ell}_X \subset X$. Then we may find a rational curve $\ell_0 \subset Y$ and its minimal birational section $\tilde{\ell_0}_X \subset X$ of $\varphi$ over $\ell_0$ which satisfy
\begin{itemize}
\item $-K_Y\cdot \ell_0=\iota_Y$ and 
\item $\deg_{(-K_X)}\tilde{\ell_0}_X=\min\left\{\deg_{(-K_X)}\tilde{\ell}_X\mid -K_Y\cdot \ell=\iota_Y \right\}$. 
\end{itemize}  
Let $\cM \subset \rat^n(X)$ be a family of rational curves containing $[\tilde{\ell_0}_X]$.
\begin{proposition}\label{prop:key} Under the above setting, we have the following:
\begin{enumerate}
\item $\cM$ is unsplit. 
\item There exists a smooth geometric quotient $\psi: X \to Z$ for $\cM$.
\end{enumerate}
\end{proposition}

\begin{proof} (1)  Assume the contrary; then, by Lemma~\ref{lem:degeneration:curves}, we might find a rational $1$-cycle $Z=\sum_{i=1}^s a_i Z_i$ as a degeneration of $\cM$-curves, where each $a_i$ is a positive integer and each $Z_i$ is a rational curve. Remark that $a_1\geq 2$ provided that $s=1$. Since we have 
$$[\ell_0]=\varphi_{\ast}(\tilde{\ell_0}_X)=\sum_{i=1}^s a_i [\varphi_{\ast}(Z_i)] \in N_1(Y),
$$there exists at least one $Z_i$ such that $\varphi(Z_i)$ is not contracted by $\varphi$. Without loss of generality, we may assume $Z_1$ is not contracted by $\varphi$ and the anticanonical degree of $Z_1$ is minimal among such curves. Then we prove the following:%Remark that $a_1>1$ provided that $s=1$.  
\begin{enumerate}
\item[(3)] $a_1=\deg (\varphi|_{Z_1})=1$, and 
\item[(4)] for any $i\neq 1$, $Z_i$ is contracted by $\varphi$. 
\end{enumerate}
To prove this, assume the contrary; then we have
\begin{align*}
\iota_Y= -K_Y\cdot \ell_0 = \sum_{i=1}^s a_i(-K_Y\cdot  \varphi_{\ast}(Z_i))\geq 2\iota_Y.\tag{5}
\end{align*}
This is a contradiction. Thus (3) and (4) hold. 
Moreover we see that $-K_Y\cdot \varphi(Z_0)=\iota_Y$. It follows from (3) that $Z_1$ is a minimal birational section over $\varphi(Z_1)$. 
One can show that $s=1$. In fact, if $s>1$, then we have 
$$
-K_X\cdot Z_1<\sum_{i=1}^s -K_X\cdot Z_i=-K_X\cdot \tilde{\ell}_0.
$$
This contradicts to the minimality of the anticanonical degree of $\tilde{\ell_0}_X$. Hence we obtain $s=1$ and $a_1=1$. However this is a contradiction. As a consequence, $\cM$ is unsplit as desired.

(2)  By Corollary~\ref{cor:geom:quot:unsplit}, it is enough to prove that $\cM$ is an unsplit covering family. 
To prove this, assume otherwise; then $\cM$ is not a dominating family. This turns out that $\codim_X{\rm Locus}(\cM)\geq 1$. Moreover, by Lemma~\ref{lem:nonfree}, we have $\deg_{(-K_X)}\cM \geq n-1$. 
Let $x \in {\rm Locus}(\cM)$ be a general point. 
By \cite[II Corollary~4.21]{Kb}, the restriction of $\varphi$ to ${\rm Locus}(\cM_x)$ is a finite morphism. 
Thus, applying Proposition~\ref{prop:Ion:Wis:2}, we have 
\begin{align*}
\dim X>\dim Y\geq \dim {\rm Locus}(\cM_x)\geq -K_X\cdot \cM+\codim_X{\rm Locus}(\cM)-1 \geq n-1.%\tag{7}
\end{align*}
Then we see that $\varphi$ is a $\P^1$-bundle. 
For any rational curve $\ell \subset Y$ such that $-K_Y\cdot \ell=\iota_Y$, let $\tilde{\ell}_X$ be a minimal birational section of $\varphi$ over $\ell$. Then we have  
\begin{align*}
n-1\leq -K_X\cdot \tilde{\ell}_X=\varphi^{\ast}(-K_Y)\cdot \tilde{\ell}_X-K_{X/Y}\cdot \tilde{\ell}_X\leq -K_Y\cdot \ell=\iota_Y.\tag{6}%-K_Y\cdot \ell+(-K_{X_{\ell}/\ell})\cdot \tilde{\ell}.
\end{align*}
By Theorem~\ref{them:CMSB:DH17}, $\varphi: X \to Y$ is either a $\P^1$-bundle over $\P^{n-1}$ or a $\P^1$-bundle over $Q^{n-1}$.
Since the Brauer group of $Y$ is trivial, $\varphi: X\to Y$ is the projectivization of a rank $2$ vector bundle $E$ on $Y$. Moreover, we see that the above $\ell \subset Y$ is a line. Twisting by a suitable line bundle on $Y$, the inequality (6) yields that $E|_{\ell}$ is isomorphic to 
\[
  \begin{cases}
    \cO_{\ell}\oplus \cO_{\ell} ~\mbox{or}~\cO_{\ell}\oplus \cO_{\ell}(-1) & \mbox{if}~ Y=\P^{n-1}; \\
        \cO_{\ell}\oplus \cO_{\ell} & \mbox{if}~ Y=Q^{n-1}; \\
  \end{cases}
\]
Note that $E|_{\ell}$ is not isomorphic to $\cO_{\ell}\oplus \cO_{\ell}(-1)$ provided that $Y=\P^{n-1}$. If $E|_{\ell}$ were isomorphic to $\cO_{\ell}\oplus \cO_{\ell}(-1)$, then \cite[Main Theorem]{Sato76} would imply that $E=\cO_{\P^{n-1}}\oplus \cO_{\P^{n-1}}(-1)$. Then $X$ is the blow-up of a point at ${\P^{n}}$. This contradicts to our assumption that $X$ does not admit a birational contraction. Thus, $E|_{\ell}$ is trivial, so that $\varphi^{-1}(\ell)=\ell\times \P^1$ for any line $\ell \subset Y$. Thus $\cM$ consists of horizontal lines in $\varphi^{-1}(\ell)=\ell\times \P^1$, which are free in $X$. This contradicts to our assumption that $\cM$ is not a covering family. Thus $\cM$ is an unsplit covering family as desired. 
\end{proof}

\subsection{Rationally connected varieties with nef $\bigwedge^2 T_X$}

\begin{lemma}\label{lem:fano:cone:simplicial} Let $X$ be a smooth projective variety with nef $\bigwedge^2 T_X$. If $X$ is a Fano variety, then the Kleiman-Mori cone $\overline{NE}(X)$ is simplicial, that is, the convex hull of linearly independent rays. 
\end{lemma}

\begin{proof} If $X$ is the blow-up of the projective space at a point, it is a Fano variety with $\rho_X=2$. Thus our assertion holds. By Theorem~\ref{them:smooth:contraction} and Theorem~\ref{them:birational:contraction}, we may assume that any contraction of an extremal ray is a smooth morphism. By slightly modifying the argument as in \cite[Proposition~4-4]{MOSW}, one can obtain our assertion. We shall give the proof for the reader's convenience. 

To prove our assertion, assume the contrary; assume the existence of extremal rays $R_1,\dots, R_k$ such that  $k >\rho(X)=m$. Then we may choose a free rational curve $C_i$ of minimal anticanonical degree among those spanning the corresponding ray $R_i$. Then there {{exist}} rational numbers $a_1, \ldots, a_m$ such that 
$$[C_k]=\sum_{i=1}^{m}a_i [C_i] \in N_1(X).$$ By the extremality of $R_k$, without loss of generality, we can assume that $a_1<0$.

For any $i \neq 1$, we take a family of rational curves $\cM_i \subset \rat^n(X)$ containing $[C_i]$. Then by 
Remark~\ref{rem:contraction:unsplit}, $\cM_i$ is unsplit. 
Applying to these families \cite[1, Lemma 2.4]{CO06} and Theorem~\ref{them:birational:contraction}, the classes $[\cM_2], \ldots, [\cM_m]$ must lie in an $(m-1)$-dimensional extremal face of $NE(X)$. A supporting divisor $H$ of this face provides a contradiction: $H\cdot C_i=0$ for $i=2, \dots, m$, $H\cdot C_1>0$ so that $H\cdot C_k <0$, contradicting that $H$ is nef.
\end{proof}

\begin{theorem}\label{them:rc:fano} Let $X$ be a smooth projective variety with nef $\bigwedge^2 T_X$. If $X$ is a rationally connected variety of $n=\dim X\geq 3$, then $X$ is a Fano variety.
\end{theorem}

\begin{proof} By \cite[Theorem]{CP92}, we may assume that $n\geq 4$. Any smooth projective rationally connected variety of Picard number one is a Fano variety; thus we may assume that $\rho_X>1$. 
By Proposition~\ref{prop:smooth:mor:target:blowup}, we may assume that any contraction of a $K_X$-negative extremal ray is a smooth morphism. We proceed by induction on the Picard number $\rho_X$. 
Suppose that our assertion is proved for all varieties whose Picard number is less than $\rho_X$.
By Proposition~\ref{prop:Y}, there exists a smooth $K_X$-negative contraction $\varphi: X \to Y$ satisfying one of the following:
\begin{enumerate}
\item $Y$ is isomorphic to $\P^1$;
%\item $Y$ is isomorphic to $\P^2$;
\item $\dim Y \geq 2$ and $\rho_Y=1$.
\end{enumerate}
If $Y$ is $\P^1$, then Proposition~\ref{proposition:over:P1} implies that $T_X$ is nef; thus by \cite[Chapter 0, Proposition]{DPS} $X$ is a Fano variety. So we assume that $\dim Y>1$. Applying Proposition~\ref{prop:key}, we obtain an unsplit covering family of rational curves $\cM \subset \rat^n(X)$ which satisfies 
\begin{itemize}
\item $\varphi^{\ast}(-K_Y)\cdot \tilde{\ell}=\iota_Y$ for any $[\tilde{\ell}] \in \cM$ and 
\item there exists a smooth geometric quotient $\psi: X \to Z$ for $\cM$. 
\end{itemize}  

Note that a smooth geometric quotient $\psi: X \to Z$ for $\cM$ is a $K_X$-negative contraction of an extremal ray $\R_{\geq 0}[\cM]$; then $Z$ is a smooth projective rationally connected variety with nef $\bigwedge^2 T_Z$. 
Let $F$ be a fiber of $\varphi: X \to Y$. By Theorem~\ref{MT2}, $F$ is a smooth Fano variety with nef $T_F$ and $\rho_F=\rho_X-1$; in particular, the Kleiman-Mori cone $NE(F)$ is simplicial (see Lemma~\ref{lem:fano:cone:simplicial}), and there exist extremal rays $R_1, \ldots, R_{\rho_X-1}$ such that 
$$i_{\ast}\left(NE(F)\right)=NE(X/Y)=R_1+ R_2 +\ldots+R_{\rho_X-1},$$ 
where each $R_i$ is generated by an extremal rational curve $C_i$: $R_i=\R_{\geq 0}[C_i]$. 
By Lemma~\ref{lem:cone:fiber}, we have an injection $i_{\ast}: N_1(F) \to N_1(X)$. On the other hand, any curve contained in $F$ is not contracted by $\psi_{\ast}: N_1(X) \to N_1(Z)$. This yields that the composition $\psi_{\ast} \circ i_{\ast}: N_1(F) \to N_1(Z)$ has a trivial kernel. Since the Picard numbers of $F$ and $Z$ coincide with each other, the composition $\psi_{\ast} \circ i_{\ast}: N_1(F) \to N_1(Z)$ is an isomorphism. 

We claim that $\psi_{\ast} \circ i_{\ast}\left(NE(F)\right)=NE(Z)$. To confirm this, note that the Kleiman-Mori cone $NE(Z)$ is simplicial. In fact, this follows from the induction hypothesis, provided that $\dim Z\geq 3$. If $\dim Z<3$, then by Proposition~\ref{prop:target:surface} $Z$ is isomorphic to $\P^1$, $\P^2$ or $\P^1\times \P^1$; in particular, $NE(Z)$ is simplicial. By \cite[Theorem~2.2]{Kane18}, each extremal ray $R_i$ goes to an extremal ray of $NE(Z)$ via the pushforward $\psi_{\ast}: N_1(X) \to N_1(Z)$. This shows our claim.    

Let $C$ be a curve on $X$. Since we have 
$$[\psi_{\ast}(C)] \in NE(Z)=\psi_{\ast} \circ i_{\ast}\left(NE(F)\right)=\R_{\geq 0}[\psi_{\ast}(C_1)]+ \ldots+\R_{\geq 0}[\psi_{\ast}(C_{\rho_X-1})],$$  there exist nonnegative real numbers $b_1, \ldots, b_{\rho_X-1}$ such that 
$$[\psi_{\ast}(C)]=\sum_{i=1}^{\rho_X-1} b_i [\psi_{\ast}(C_i)].$$  
This implies that $$[C]- \sum_{i=1}^{\rho_X-1} b_i[C_i] \in {\rm Ker}(\psi_{\ast})=\langle[\tilde{\ell}_0]\rangle_{\R}.$$
Thus we have $b \in \R$ such that 
$$
[C]=\sum_{i=1}^{\rho_X-1} b_i [C_i]+b[\tilde{\ell}_0] \in N_1(X).
$$
For an ample divisor $H$ on $Y$, we have 
$$0\leq \varphi^{\ast}H\cdot C=\sum_{i=1}^{\rho_X-1} b_i \varphi^{\ast}H\cdot C_i+b\varphi^{\ast}H\cdot \tilde{\ell}_0=bH\cdot {\ell}_0.$$
Since $H\cdot {\ell}_0>0$, we see that $b\geq 0$. As a consequence, we obtain 
$$
NE(X)=\R_{\geq 0}[C_1]+\ldots +\R_{\geq 0}[C_{\rho_X-1}]+\R_{\geq 0}[\tilde{\ell}_0].
$$ 
Thus, applying Kleiman's ampleness criterion \cite[Theorem~1.8]{KM}, $-K_X$ is ample.
\end{proof}

\subsection{Conclusions} 
\begin{proof}[{Proof of Theorem~\ref{MT}}] Theorem~\ref{MT} is a direct consequence of Proposition~\ref{prop:weaker:MT} and Theorem~\ref{them:rc:fano}.
\end{proof}

\begin{proof}[{Proof of Corollary~\ref{cor:MT}}] Let $X$ be as in Corollary~\ref{cor:MT}. By Theorem~\ref{MT}, there exists a finite \'etale cover $f: \tilde{X} \to X$ such that the Albanese morphism $\alpha_{\tilde{X}}: \tilde{X} \to {\rm Alb}(\tilde{X})$ is a locally trivial fibration whose fibers are isomorphic to a smooth Fano variety $F$. If ${\rm Alb}(\tilde{X})=0$, then $\tilde{X}\cong F$ is a smooth Fano variety. Since a smooth Fano variety is simply connected, $\tilde{X}$ is isomorphic to $X$. Thus it is sufficient to prove that $T_{{X}}$ is nef provided that ${\rm Alb}(\tilde{X})>0$. 
Since $f$ is \'etale, the nefness of $T_X$ is equivalent to that of $T_{\tilde{X}}$. 
By the same argument as in the proof of Proposition~\ref{prop:fiber:target}, we obtain exact sequences  
\begin{align*}
0 \to T_{\tilde{X}/{\rm Alb}(\tilde{X})} \to T_{\tilde{X}} \to \alpha_{\tilde{X}}^{\ast} T_{{\rm Alb}(\tilde{X})}\to 0 \tag{1}.
\end{align*}
\begin{align*}
0 \to \bigwedge^2 T_{\tilde{X}/{\rm Alb}(\tilde{X})} \to E \to T_{\tilde{X}/{\rm Alb}(\tilde{X})}  \otimes \alpha_{\tilde{X}}^{\ast} T_{{\rm Alb}(\tilde{X})}\to 0 \tag{2}
\end{align*}
\begin{align*}
0 \to E \to \bigwedge^2T_{\tilde{X}} \to \alpha_{\tilde{X}}^{\ast} \left(\bigwedge^2T_{{\rm Alb}(\tilde{X})}\right)\to 0 \tag{3}
\end{align*} for some vector bundle $E$ on $X$. 
Remark that the tangent bundle $T_{{\rm Alb}(\tilde{X})}$ is trivial; thus the exact sequence (3) and Proposition~\ref{prop:bundle:nef} (5) tell us that $E$ is nef. Combining with Proposition~\ref{prop:bundle:nef} (1), this implies that $T_{\tilde{X}/{\rm Alb}(\tilde{X})}$ is nef. Then, applying Proposition~\ref{prop:bundle:nef} (5) again, we see that $T_{\tilde{X}}$ is nef. As a consequence, we see that $T_{{X}}$ is nef provided that ${\rm Alb}(\tilde{X})>0$.
\end{proof}

\begin{proof}[{Proof of Theorem~\ref{MT2}}] (1) is nothing but Theorem~\ref{them:birational:contraction}; thus we check (2). Assume that $X$ does not admit a birational contraction of a $K_X$-negative extremal ray. Then by Theorem~\ref{them:smooth:contraction} any contraction of a $K_X$-negative extremal ray $\varphi: X \to Y$ is smooth.\end{proof}

\begin{proof}[{Proof of Theorem~\ref{them:rc:fano:0}}] Theorem~\ref{them:rc:fano:0} is a direct consequence of Lemma~\ref{lem:fano:cone:simplicial} and Theorem~\ref{them:rc:fano}.
\end{proof}

To end this section, we give a structure theorem of varieties with nef $\bigwedge^2 T_X$:
\begin{theorem}\label{them:further:study} Let $X$ be a smooth projective variety with nef $\bigwedge^2 T_X$ and $n=\dim X \geq 3$. Then either $T_X$ is nef or $X$ is one of the following:
\begin{enumerate}
\item $X$ is the blow-up of the projective space $\P^n$ at a point; 
\item $X$ is a Fano variety of pseudoindex $n-1$ and Picard number $1$;
\item $X$ is a Fano variety of Picard number $1$ which satisfies the following:
\begin{itemize}
\item[(*)] there exists a minimal rational component $\cM \subset \rat^n(X)$ such that $\cM$ is unsplit and any $\cM$-curve is free;
\end{itemize}
\item $X$ is a Fano variety of Picard number $m>1$ such that its Kleiman-Mori cone is simplicial: $NE(X)=R_1+\ldots+R_{m}$. Moreover given any proper subset $I \subset \{1,2, \ldots, m\}$, we denote by $\varphi_I: X \to X_I$ a contraction of an extremal face $\sum_{i\in I}R_i$. Then the contraction $\varphi_I: X \to X_I$ satisfies the following:
\begin{enumerate}
\item $\varphi_I$ is a smooth morphism whose fibers are Fano varieties with nef tangent bundle and Picard number $\sharp I$;
\item if $\sharp I\leq m-2$, then $X_I$ is a smooth Fano variety with nef $\bigwedge^2 T_{X_I}$ and Picard number $m-\sharp I$ such that its Kleiman-Mori cone is described as $NE(X_I)=\sum_{i\not\in I}(\varphi_I)_{\ast}R_i$;
\item if $\sharp I= m-1$, then $X_{I}$ is a smooth Fano variety of Picard number $1$ which satisfies the above condition $(\ast)$ or $\iota_{X_{I}}=\dim X_{I}-1$.
\end{enumerate}
\end{enumerate}
\end{theorem}

\begin{proof} By Corollary~\ref{cor:MT}, we may assume that $X$ is a Fano variety. 
Suppose that $\rho_X=1$. Taking a rational curve $C \subset X$ whose anticanonical degree is equal to the pseudoindex $\iota_X$, let $\cM \subset \rat^n(X)$ be a family of rational curves containing $[C]$. By Lemma~\ref{lem:degeneration:curves}, $\cM$ is unsplit. If any $\cM$-curve is free, then $X$ satisfies the condition $(\ast)$. Otherwise there is a non-free $\cM$-curve. Then Lemma~\ref{lem:nonfree} implies that $\iota_X=\deg_{(-K_X)}\cM \geq n-1$. By Theorem~\ref{them:CMSB:DH17}, we see that either $X$ is $\P^n$, $Q^n$ or a Fano variety with $\iota_X = n-1$. Since $\P^n$ and $Q^n$ have nef tangent bundle, our assertion holds. 

We assume that $\rho_X>1$. By Theorem~\ref{MT2} (1), if $X$ admits a birational contraction of an extremal ray, then $X$ is isomorphic to the blow-up of $\P^n$ at a point. Hence we assume that any contraction of an extremal ray is of fiber type; applying Theorem~\ref{MT2} (2), it is a smooth morphism. By Lemma~\ref{lem:fano:cone:simplicial}, the Kleiman-Mori cone is simplicial: $NE(X)=R_1+\ldots+R_{m}$. Given any proper subset $I \subset \{1,2, \ldots, m\}$, Mori's cone theorem tells us that there is a contraction of an extremal face $\sum_{i\in I}R_i$, which is denoted by $\varphi_I: X \to X_I$. We claim that $\varphi_I: X \to X_I$ satisfies (a)-(c). Indeed, by using Theorem~\ref{MT2} (2), Proposition~\ref{prop:fiber:target} (1), Proposition~\ref{prop:smooth:mor:target:blowup} and Lemma~\ref{lem:fano:cone:simplicial} repeatedly, one can show that $\varphi_I$ can be described as a composition of smooth contractions of an extremal ray; thus it is a smooth morphism. Then (a) and (b) follow from Lemma~\ref{lem:cone:fiber} and Proposition~\ref{prop:fiber:target}. To prove (c), assume that $\sharp I= m-1$. Then by Lemma~\ref{lem:cone:fiber} $X_{I}$ is a smooth Fano variety of Picard number $1$. We may assume $\dim X_I>2$; otherwise $X_I$ is $\P^1$ or $\P^2$ which in turn implies $X_I$ satisfies the condition $(\ast)$. Thus $X_I$ is a smooth Fano variety with nef $\bigwedge^2 T_{X_I}$, $\rho_{X_I}=1$ and $\dim X_I \geq 3$. As we have seen in the former part of this proof, we see that such $X_{I}$ satisfies either the condition $(\ast)$ or $\iota_{X_{I}}=\dim X_{I}-1$.  

\if0
proposition~\ref{prop:fiber:target} proposition~\ref{prop:smooth:mor:target:blowup}  proposition~\ref{prop:target:surface} proposition~\ref{prop:Y} , we see that $X$ satisfies the assumption of 

Theorem~\ref{them:rho=1:unsplit:free} and Proposition~\ref{prop:target:free:iota}, we see that $X$ satisfies the assumption of Theorem~\ref{them:kane:key}. Hence Theorem~\ref{them:kane:key} implies that $X$ is a rational homogeneous variety; thus $T_X$ is nef.

 Lemma~\ref{lem:cone:fiber}, Proposition~\ref{prop:fiber:target}, Proposition~\ref{prop:smooth:mor:target:blowup} and Proposition~\ref{prop:target:surface}.

\fi
\end{proof}

\section{Special varieties}

\subsection{Toroidal case}

\begin{definition}\label{def:spherical}
\begin{enumerate}
\item Let $G$ be a reductive linear algebraic group and $B$ a Borel subgroup of $G$. A $G$-variety $X$ is {\it ($G$-)spherical} if it has a dense $B$-orbit. A spherical $G$-variety $X$ is {\it ($G$-)toroidal} if every $B$-stable but not $G$-stable divisor contains no $G$-orbit. 
\item Let $G$ be a connected algebraic group and $X$ a smooth $G$-variety; let $D\subset X$ be a $G$-stable effective reduced divisor with normal crossings. We denote by $T_X(-\log D)$ the {\it sheaf of logarithmic vector fields} which is by definition the subsheaf of the tangent sheaf $T_X$ consisting of derivations that preserve the ideal sheaf of $D$. We say that $X$ is {\it log homogeneous with boundary $D$} if the logarithmic tangent bundle $T_X(-\log D)$ is generated by its global sections. We say that $X$ is {\it log homogeneous} if $X$ is {log homogeneous with some boundary $D$}.
\end{enumerate}
\end{definition}

Remark that any smooth projective toric variety is toroidal. Furthermore we have the following:

\begin{proposition}[{\cite[Proposition~2.2.1]{Bri07} and \cite[Corollary 2.1.4 and Corollary 3.2.2]{BB96}}]\label{prop:toroidal:loghomog} Let $G$ be a reductive linear algebraic group and $X$ a smooth complete $G$-variety. Then X is toroidal if and only if it is log homogeneous.
\end{proposition}

\begin{proof}[Proof of Theorem~\ref{them:Schmitz:generalization}] We assume that $X$ is not isomorphic to the blow-up of a projective space at a point. According to the result of Q. Li {\cite[Theorem~1.2]{Li}}, it is enough to prove that any curve on $X$ is nef as a cycle. By Theorem~\ref{them:further:study}, the Kleiman-Mori cone $NE(X)$ can be described as follows:
$$
NE(X)=\R_{\geq 0}[C_1]+ \R_{\geq 0}[C_2]+\ldots+ \R_{\geq 0}[C_{\rho_X}], 
$$ where each $C_i$ is an extremal rational curve.
Moreover a contraction of each ray $R_i$ is smooth; this concludes that each extremal rational curve $C_i$ is free. In particular, $C_i$ is nef. As a consequence, any curve on $X$ is nef as desired.
\end{proof}

\subsection{The case $\dim X \leq 6$}
In this subsection, we shall prove Theorem~\ref{them:CP:second:wedge:generalization}.
We begin with recalling some results on the Campana-Peternell conjecture (=Conjecture~\ref{conj:CP}):
\begin{theorem}\label{them:rho=1:unsplit:free} Let $X$ be a smooth Fano variety of $\dim X=n$ and $\rho_X
=1$. Assume $X$ admits a minimal rational component $\cM \subset \rat^n(X)$ such that $\cM$ is unsplit and any $\cM$-curve is free. Assume moreover either
\begin{enumerate}
\item $n \leq 5$ or 
\item $n=6$ and $\iota_X\neq 5$.
\end{enumerate}
Then $X$ is a rational homogeneous variety. 
In particular, the Campana-Peternell conjecture holds for the cases (1) and (2).
\end{theorem}

\begin{proof} This follows from arguments as in \cite{CMSB, DH17, Mk, Hw4, Kane}. We sketch the proof for the reader's convenience. Since the later follows from the former, we only explain the former part. Under the assumption of the theorem, we have an associated universal family:
 \[
  \xymatrix{
   \cU \ar[r]^q \ar[d]_p &  X \\
    \cM &   }
\]
We denote by $d$ the anticanonical degree of the family $\cM$. By Mori's bend and break lemma, $d$ is at most $n+1$. Furthermore, the same proofs as in \cite{CMSB} (see also \cite{Ke}) and \cite{DH17} shows that $X$ is isomorphic to $\P^n$ or a quadric, provided that $d\geq n$. Thus we may assume that $d\leq n-1$. 

Since any $\cM$-curve is free, one obtain that $q$ is a smooth morphism and $\cM$ is a smooth variety of dimension $n+d-3$ by \cite[II. Theorem~1.7, Theorem~2.15, Corollary~3.5.3]{Kb}; moreover %\cite[II. Proposition~2.14.1]{Kb} tells us that $\cM$ is projective, and 
\cite[II. Corollary~2.12]{Kb} implies that $p$ is a $\P^1$-bundle. Since the dimension of ${\cM}_x$ is non-negative, we see that $d$ is at least $2$. We claim that $d$ is greater than $2$. If $d= 2$, then 
 $q$ would be \'etale; however $q$ should be an isomorphism, because a smooth Fano variety is simply connected. Thus we have $3\leq d$.
 
Suppose that $d=3$. 
The hyperbolicity of the moduli space of curves yields that $q$ is a $\P^1$-bundle (see \cite[Lemma~1.2.2]{Mk}). Thus the universal family $\cU$ admits two $\P^1$-bundle structures, and the result follows from \cite{OSWW} (or more generally \cite[Theorem~1.1]{OSWi}). Since the case where $d=4$ is more complicated, so we omit the details. We refer the reader to  \cite{Kane} and \cite[Theorem~2.4]{Wat20}. 
\end{proof}

%The following two results are key to prove Theorem~\ref{them:CP:second:wedge:generalization}:
\begin{theorem}[{\cite[Theorem~0.2]{Kane19}}]\label{them:kane:key} Let $X$ be a smooth Fano variety with dimension $n$ and $\rho_X>n-6$. For any sequence of contractions of an extremal ray
$$
X\xrightarrow{f_1}X_1\xrightarrow{f_2}X_2\xrightarrow{f_3} \ldots \xrightarrow{f_{m-1}}X_{m-1}\xrightarrow{f_m}X_m,
$$
assume that each $f_i$ is a rational homogeneous bundle. Then $X$ is either
\begin{enumerate}
\item a rational homogeneous variety or
\item $(\P^1)^{n-7}\times X_0$, where $X_0$ is a Fano $7$-fold with Picard number $2$ (we omit the detailed description of $X_0$).
 \end{enumerate}   
\end{theorem}

\begin{proposition}\label{prop:target:free:iota} Let $X$ be a smooth projective variety with nef $\bigwedge^2 T_X$ and $F$ a rational homogeneous variety. Assume that $X$ admits an $F$-bundle structure $\varphi: X \to Y$ onto a smooth Fano variety $Y$ with $\dim Y\geq 2$ and $\rho_Y=1$. Let $\cM \subset \rat^n(Y)$ be a family of rational curves with $\deg_{(-K_Y)}\cM=\iota_Y$. Then $\cM$ is an unsplit covering family such that any $\cM$-curve is free.
\end{proposition}

\begin{proof} By Lemma~\ref{lem:degeneration:curves} and the minimality of $\iota_Y$, it is enough to prove any rational curve $\ell \subset Y$ with $-K_Y\cdot \ell=\iota_Y$ is free. Consider an $F$-bundle $\varphi_{\ell}: X_{\ell} \to \ell=\P^1$ and its section $\tilde{\ell} \subset X_{\ell}$ as in Definition~\ref{def:section}. Then by \cite[Lemma~2.3]{OSWi} (see also \cite[Remark~3.18]{OSWi}), we obtain $-K_{X_{\ell}/\P^1}\cdot \tilde{\ell} \leq 0$. Then it turns out
\begin {align*}
-K_X\cdot \tilde{\ell}_X&=-K_{X/Y}\cdot \tilde{\ell}_X+(-K_Y)\cdot \ell=-K_{X_{\ell}/\P^1}\cdot \tilde{\ell}+(-K_Y)\cdot \ell \leq -K_Y \cdot \ell=\iota_Y. \nonumber
\end {align*}
If $\iota_Y \leq \dim X-2$, then it follows from Lemma~\ref{lem:nonfree} that $\tilde{\ell}_X$ is free in $X$; thus $\ell$ is free in $Y$. If $\iota_Y \geq \dim X-1\geq \dim Y$, then by Theorem~\ref{them:CMSB:DH17} we see that $Y$ is isomorphic to a quadric or a projective space; in particular, $Y$ is homogeneous. Hence our assertion holds.
\end{proof}

\begin{proof}[Proof of Theorem~\ref{them:CP:second:wedge:generalization}] By Theorem~\ref{them:further:study}, we may assume that $X$ satisfies either (3) or (4) in Theorem~\ref{them:further:study}. Then Theorem~\ref{them:rho=1:unsplit:free}, Theorem~\ref{them:kane:key} and Proposition~\ref{prop:target:free:iota} imply that $X$ is a rational homogeneous variety; thus $T_X$ is nef.  
\end{proof}

Similar to Theorem~\ref{them:rho=1:unsplit:free} and Theorem~\ref{them:kane:key}, we predicts most varieties $X$ satisfing either (3) or (4) in Theorem~\ref{them:further:study} admit a nef tangent bundle. Then it is natural to ask the following:
\begin{problem}\label{prob:last} Let $X$ be a smooth Fano variety with nef $\bigwedge^2 T_X$ and $n=\dim X \geq 3$. Assume that $T_X$ is not nef. Then is $X$ a Fano variety with $\rho_X=1$ and $\iota_X=n-1$?
\end{problem}

{\bf Acknowledgements.} 
Some parts of this paper were influenced by the author's joint work with Sho Ejiri and Akihiro Kanemitsu on the study of varieties with nef tangent bundle in positive characteristic. The author would like to thank them for fruitful discussions. The author is grateful to Taku Suzuki for his careful reading of the earlier draft and for his valuable comments. Thanks to his comments, the proof of Proposition~\ref{prop:key} was simplified.

\bibliographystyle{plain}
\bibliography{biblio}
\end{document}